\documentclass{amsart}
\usepackage{amssymb,amstext,amsmath,amscd,amsthm,amsfonts,enumerate,graphicx,latexsym,stmaryrd,multicol}
\usepackage[usenames]{color}
\usepackage{setspace}
\usepackage{bm}
\usepackage[all]{xy}

\newtheorem{thm}{Theorem}[section]
\newtheorem{lem}[thm]{Lemma}
\newtheorem{prop}[thm]{Proposition}
\newtheorem{cor}[thm]{Corollary}
\theoremstyle{definition}
\newtheorem{dfn}[thm]{Definition}
\newtheorem{ques}[thm]{Question}
\newtheorem{rem}[thm]{Remark}

\theoremstyle{remark}

\newtheorem*{proof of claim}{Proof of Claim}

\numberwithin{equation}{thm}
\def\C{\mathcal{C}}
\def\cok{\operatorname{Coker}}
\def\depth{\operatorname{depth}}
\def\fd{\operatorname{fd}}
\def\Gfd{\operatorname{Gfd}}
\def\Gid{\operatorname{Gid}}
\def\Gpd{\operatorname{Gpd}}
\def\Hom{\operatorname{Hom}}
\def\id{\operatorname{id}}
\def\image{\operatorname{Im}}
\def\ker{\operatorname{Ker}}
\def\L{\mathbf{L}}
\def\m{\mathfrak{m}}
\def\p{\mathfrak{p}}
\def\pd{\operatorname{pd}}
\def\RHom{\operatorname{\mathbf{R}Hom}}
\def\spec{\operatorname{Spec}}
\def\Tor{\operatorname{Tor}}
\def\tr{\operatorname{Tr}}
\def\xx{\boldsymbol{x}}
\begin{document}
\allowdisplaybreaks
\title{Complexes of finite Gorenstein flat and injective dimensions}
\author{Kaito Kimura}
\address{Department of Mathematics, Purdue University, West Lafayette, IN 47907, USA}
\email{m21018b@gmail.com}
\thanks{The author was supported by JSPS Overseas Research Fellowships.}

\begin{abstract}
In this paper, we consider a Gorenstein-dimensional analogue of Foxby's characterization of Gorenstein rings. We prove that a commutative Noetherian local ring is Gorenstein if it admits a complex whose depth, Gorenstein flat dimension, and Gorenstein injective dimension are all finite. This gives an affirmative answer to the original question of Christensen, Foxby, and Holm, which had remained open in this generality even for modules, and at the same time establishes its natural extension to complexes.
\end{abstract}
\maketitle
\section{Introduction}

Throughout the present paper, all rings are assumed to be commutative and Noetherian. 
A number of classical results in commutative algebra characterize local rings by the existence of modules or complexes with suitable homological finiteness properties. Bass' conjecture gives a fundamental instance: if a local ring admits a non-zero finitely generated module of finite injective dimension, then it is Cohen--Macaulay; see \cite[Corollary 9.6.2 and Remark 9.6.4]{BH}. Foxby \cite{Fox, Fox2} proved that a local ring is Gorenstein if it admits a complex whose depth, flat dimension, and injective dimension are all finite.

Gorenstein projective, Gorenstein flat, and Gorenstein injective dimensions refine their classical counterparts. They provide a framework for extending the homological methods and favorable properties familiar over Gorenstein rings to modules and complexes over rings that need not be Gorenstein. Moreover, the finiteness of these Gorenstein dimensions leads to rich homological phenomena, including equivalences of categories in the derived category induced by dualizing complexes. In particular, many results concerning modules or complexes of finite projective, flat, or injective dimension naturally suggest corresponding Gorenstein versions. It is therefore natural to ask to what extent the classical characterizations discussed above remain valid when these classical homological dimensions are replaced by their Gorenstein analogues. For example, a Gorenstein injective analogue of Bass' conjecture was proposed by Takahashi \cite{Tak}. Christensen, Foxby, and Holm \cite{ChFoHo} posed the following question and proved that it has an affirmative answer when at least one of the flat dimension and the injective dimension of the module is finite.

\begin{ques}\cite[Question 4.29]{ChFoHo}\label{quesmod}
Must a local ring be Gorenstein if it admits a module whose depth, Gorenstein flat dimension, and Gorenstein injective dimension are all finite?
\end{ques}

This question can be viewed as a Gorenstein-dimensional analogue of Foxby's characterization of Gorenstein local rings, in which the existence of a complex is replaced by the existence of a module. There has been recent progress on this problem. Azer\^{e}do, Miranda-Neto, and Souza \cite{AMNS} used Foxby classes associated to semidualizing modules to prove an affirmative answer in the Cohen--Macaulay case under an additional non-vanishing condition on the residue fiber. More precisely, they showed that if $(R, \m, k)$ is Cohen--Macaulay and there exists an $R$-module $M$ such that $M\otimes_R k\ne 0$ and both the Gorenstein flat dimension and the Gorenstein injective dimension of $M$ are finite, then $R$ is Gorenstein. (The condition $M\otimes_R k\ne 0$ is stronger than the finiteness of the depth of $M$.) The purpose of this paper is to prove a complex version of this statement without assuming that $R$ is Cohen--Macaulay. As a consequence, we obtain the following result.

\begin{thm} [Corollary \ref{main cor}]
A local ring admitting a complex whose depth, Gorenstein flat dimension, and Gorenstein injective dimension are all finite is Gorenstein.
\end{thm}

This gives an affirmative answer to Question \ref{quesmod}, and more generally to its complex version. The strategy is to reduce the problem to showing that a certain finitely generated module arising from a free resolution of a dualizing complex is free. The key lemma needed for this reduction is a generalization of the argument that played an essential role in the proof of the theorem of Azer\^{e}do, Miranda-Neto, and Souza. 
The notation and all the results of this paper are presented in the next section.

\section{Results}

We first introduce the notation that will be used throughout this paper.

\begin{dfn}
Let $R$ be a ring and $X=(\cdots \to X_n\xrightarrow{\partial_n} X_{n-1} \xrightarrow{\partial_{n-1}} \cdots)$ a complex of $R$-modules.
We denote isomorphisms by the symbol $\cong$ and quasi-isomorphisms by the symbol $\simeq$.
We omit subscripts/superscripts if there is no ambiguity.

(1) For each integer $m$, $X[m]$ denotes the complex $X$ shifted $m$ degrees, defined by $(X[m])_n=X_{n-m}$. A complex $X$ is called \textit{bounded to the left} if $X_n=0$ for all $n\gg 0$; similarly, $X$ is called \textit{bounded to the right} if $X_n=0$ for all $n\ll 0$. A complex which is bounded to the left and to the right is said to be \textit{bounded}. When $X$ is quasi-isomorphic to a bounded complex, we say that $X$ is \textit{homologically bounded}, and when all homology modules of $X$ vanish, we say that $X$ is \textit{exact} (or \textit{acyclic}). The category of complexes of $R$-modules is denoted by $\C(R)$.
Given a full subcategory $\mathcal{A}(R)$ of $\C(R)$, we write $\mathcal{A}_{\sqsubset}(R)$, $\mathcal{A}_{\sqsupset}(R)$, and $\mathcal{A}_{\square}(R)$ for the subcategories of $\mathcal{A}(R)$ consisting of complexes that are bounded to the left, bounded to the right, and bounded, respectively.

(2) Let $M$ be an $R$-module.
If there exists an exact complex $X$ of projective $R$-modules such that $M\cong \image \partial_0$ and the complex $\Hom_R(X,P)$ is exact for any projective $R$-module $P$, then $M$ is called \textit{Gorenstein projective}.
If there exists an exact complex $X$ of flat $R$-modules such that $M\cong \image \partial_0$ and the complex $I\otimes_R X$ is exact for any injective $R$-module $I$, then $M$ is said to be \textit{Gorenstein flat}.
If there exists an exact complex $X$ of injective $R$-modules such that $M\cong \image \partial_0$ and the complex $\Hom_R(I,X)$ is exact for any injective $R$-module $I$, then $M$ is said to be \textit{Gorenstein injective}.
We denote by $\C^\mathrm{P}(R)$, $\C^\mathrm{F}(R)$, $\C^\mathrm{I}(R)$, 
$\C^\mathrm{GP}(R)$, $\C^\mathrm{GF}(R)$, and $\C^\mathrm{GI}(R)$ the full subcategories of $\C(R)$ consisting of complexes whose terms are projective, flat, injective, Gorenstein projective, Gorenstein flat, and Gorenstein injective modules, respectively.
Projective, flat, and injective modules are also Gorenstein projective, Gorenstein flat, and Gorenstein injective, respectively.

(3) We define the following homological dimensions:
\begin{itemize}
\item $\pd_R X:={\rm inf}\{{\rm sup}\{n \mid P_n\ne 0\} \mid P\simeq X\in \C_{\sqsupset}^\mathrm{P}(R)\}$ is the \textit{projective dimension};
\item $\fd_R X:={\rm inf}\{{\rm sup}\{n \mid F_n\ne 0\} \mid F\simeq X\in \C_{\sqsupset}^\mathrm{F}(R)\}$ is the \textit{flat dimension};
\item $\id_R X:={\rm inf}\{{\rm sup}\{n \mid I_{-n}\ne 0\} \mid X\simeq I\in \C_{\sqsubset}^\mathrm{I}(R)\}$ is the \textit{injective dimension};
\item $\Gpd_R X:={\rm inf}\{{\rm sup}\{n \mid P_n\ne 0\} \mid P\simeq X\in \C_{\sqsupset}^\mathrm{GP}(R)\}$ is the \textit{Gorenstein projective dimension};
\item $\Gfd_R X:={\rm inf}\{{\rm sup}\{n \mid F_n\ne 0\} \mid F\simeq X\in \C_{\sqsupset}^\mathrm{GF}(R)\}$ is the \textit{Gorenstein flat dimension};
\item $\Gid_R X:={\rm inf}\{{\rm sup}\{n \mid I_{-n}\ne 0\} \mid X\simeq I\in \C_{\sqsubset}^\mathrm{GI}(R)\}$ is the \textit{Gorenstein injective dimension}.
\end{itemize}
A complex with one of these homological dimensions finite is homologically bounded. When $(R,\m,k)$ is local, the \textit{depth} of $X$ is defined as follows: $$\depth_R X=-{\rm sup}\{i\mid H_i(\RHom_R(k,X))\ne 0\}.$$

(4) We set $(-)^\ast=\Hom_R(-,R)$. For an $R$-module $M$, denote by $E_R(M)$ the injective hull of $M$. When $R$ is local, and $M$ is finitely generated, for a minimal free resolution $(\cdots\to F_1\xrightarrow{d}F_0\to0)$ of $M$, the \textit{(Auslander) transpose} $\tr M$ of $M$ is defined as $\cok(d^\ast: (F_0)^\ast \to (F_1)^\ast)$. Throughout this paper, whenever we say that $D$ is a dualizing complex over $R$, we take $D$ to be of the following form:
$$
D_{-i}\cong \bigoplus_{\substack{\p\in\spec R \\ \dim R/\p=d-i}} E_R(R/\p) \quad \text{for all}\ 0\le i\le d, \ \text{where}\ d=\dim R.
$$
\end{dfn}

The Auslander and Bass classes characterize, respectively, complexes of finite Gorenstein flat dimension and complexes of finite Gorenstein injective dimension in the derived category, and the dualizing complex induces an equivalence between these two classes; see \cite{CFrH} for instance. For the purposes of this paper, we record only the consequences of this theory that will be used later.

\begin{rem}\label{Well-known fact}
Let $R$ be a ring with a dualizing complex $D$, and let $X$ be a homologically bounded complex of $R$-modules. It follows from \cite[1.3, and Theorems 4.1 and 4.4]{CFrH} that $X$ has finite (Gorenstein) flat dimension if and only if $X\otimes_R^{\L} D$ has finite (Gorenstein) injective dimension, and that $X$ has finite (Gorenstein) injective dimension if and only if $\RHom_R(D,X)$ has finite (Gorenstein) projective dimension. This correspondence is called \textit{Foxby equivalence}. For each of the homological dimensions considered here, finiteness satisfies the usual two-out-of-three property with respect to exact triangles; see \cite[Propositions 9.2.15 and 9.3.25]{CFoH}. In particular, by induction on the number of non-zero components of the complex, tensoring with a bounded complex of free modules preserves the finiteness of (Gorenstein) homological dimensions.
\end{rem}

This paper is motivated by the following result of Foxby and by Question \ref{quesmod} which is widely recognized, but has remained open since its formulation.

\begin{prop}\cite[Propositions 2.8 and 2.10]{Fox2}\label{fdid}
Let $R$ be a local ring. If there exists a complex $X$ of $R$-modules with $\depth_R X<\infty$, $\fd_R X<\infty$, and $\id_R X<\infty$, then $R$ is Gorenstein.
\end{prop}

Azer\^{e}do, Miranda-Neto, and Souza gave an affirmative answer to Question \ref{quesmod}, under mild additional assumptions, in the case where the ring is Cohen--Macaulay. Although their proof does not explicitly invoke \cite[Lemma 2.3]{Kim}, its essential argument is, in effect, parallel to that lemma: it uses the same mechanism to show that the first syzygy of the canonical module is free. We adapt this strategy to the setting of dualizing complexes. This, however, is far from a formal analogy. To overcome a serious obstruction that arises in the proof, we first generalize \cite[Lemma 2.3]{Kim}. Indeed, that lemma is recovered from the following one by taking $L=M\otimes_R N$ and letting $\pi$ be the identity map.

\begin{lem}\label{apdtfreelem}
Let $(R,\m,k)$ be a local ring, and let $M$ be a finitely generated $R$-module. Then $M$ is free if there exist $R$-modules $L$ and $N$ satisfying the following conditions:
\begin{enumerate}[\rm(1)]
\item There is a homomorphism $\pi: L\to M\otimes_R N$, and one has $N\otimes_R k\ne 0$ and $\Tor_1^R(\tr M, L)=0$;
\item For every $n\in N$ and the natural homomorphism $\phi: M \to M\otimes_R N: m\mapsto m\otimes n$, there exists $\psi: M\to L$ such that $\phi=\pi \circ\psi$.
\end{enumerate}
\end{lem}

\begin{proof}
We prove the lemma by induction on the minimal number of generators $\mu$ of $M$. The case $\mu=0$ is obvious. Suppose $\mu>0$. Let $I$ be the image of the evaluation map $\lambda: M\otimes_R M^\ast\to R: m\otimes f\to f(m)$. There exists a natural exact sequence $M^\ast\otimes_R L \xrightarrow{\alpha}\Hom_R(M,L) \to \Tor_1^R(\tr M, L)\to 0$, where $\alpha(f\otimes l)(m)=f(m)l$ for $f\in M^\ast, l\in L, m\in M$; see \cite[Proposition 2.8(b)]{AB}. As $\Tor_1^R(\tr M, L)=0$, $\alpha$ is surjective. 

We claim that $M\otimes_R N=I(M\otimes_R N)$. For every $n\in N$ and the natural homomorphism $\phi: M \to M\otimes_R N: m\mapsto m\otimes n$, there exists $\psi: M\to L$ such that $\phi=\pi \circ\psi$. Since $\alpha$ is surjective, there are $f_1,\ldots, f_s\in M^\ast$ and $l_1,\ldots, l_s\in L$ such that $\psi=\alpha(\sum_{j=1}^s f_j \otimes_R l_j)$, which means that $\psi(m)=\sum_{j=1}^s f_j(m)l_j$ for all $m\in M$. By this, for each $m\in M$, we have $$f_1(m),\ldots,f_s(m) \in I {\rm \ and \ } m\otimes n=\phi(m)=\pi(\psi(m))=\sum_{j=1}^s f_j(m)\pi(l_j)\in I (M\otimes_R N).$$
If $I\subseteq \m$, then there exists a surjection $R/I\otimes_R M\otimes_R N \to k\otimes_R M\otimes_R N \cong k^{\oplus \mu}\otimes_R N$, contradicting the assumptions $\mu>0$ and $N\otimes_R k\ne 0$  since $R/I\otimes_R M\otimes_R N\cong (M\otimes_R N)/I(M\otimes_R N)=0$. Hence $I=R$, which implies that there is $f\in M^\ast$ such that $\image f \nsubseteq \m$. Then $f$ is surjective, and thus it is split. One has $M\cong M'\oplus R$ for $M'=\ker f$. By the definition of the transpose, one has $\tr M\cong\tr M'$ up to free summands, and thus $\Tor_1^R(\tr M', L)=0$. For the natural split epimorphism $\rho: M\to M'$, we put $\pi'=(\rho\otimes \id_N) \circ \pi$. Let $n\in N$. For the natural homomorphism $\phi: M \to M\otimes_R N: m\mapsto m\otimes n$, there exists $\psi: M\to L$ such that $\phi=\pi \circ\psi$ by assumption.  For the natural split monomorphism $\iota: M'\to M$, the composition $\psi \circ \iota: M' \to L$ satisfies 
$$(\pi'\circ(\psi \circ \iota))(m')=((\rho\otimes \id_N) \circ \pi\circ\psi \circ \iota)(m')=((\rho\otimes \id_N)\circ \phi \circ \iota)(m')=m' \otimes n.$$
From the above, $L, M', N$, and $\pi'$ also satisfy the assumptions of the lemma. Since the minimal number of generators of $M'$ is $\mu-1$, by the induction hypothesis, $M'$ is free, and thus so is $M$.
\end{proof}

We now state the main theorem of this paper.

\begin{thm}\label{GfdGid}
Let $(R,\m,k)$ be a local ring of dimension $d$ with a dualizing complex $D$. If there exists a non-acyclic complex $X$ of $R$-modules with $\Gfd_R X<\infty$, $\Gid_R X<\infty$, and $H_n(X)\otimes_R k\ne 0$ for $n={\rm inf}\{i \mid H_i(X)\ne 0\}$, then $R$ is Gorenstein.
\end{thm}

\begin{proof}
After applying a suitable shift, we may assume $n=0$. Then we can take a projective resolution $P=(\cdots \to P_1\to P_0\to 0)$ of $X$, and regard $(-)\otimes_R P$ as a model for $(-)\otimes_R^{\L} X$.

Let $F=(\cdots \to F_m \xrightarrow{\partial_{m}^F} F_{m-1}  \xrightarrow{\partial_{m-1}^F}  \cdots \xrightarrow{\partial_{-d+2}^F}  F_{-d+1} \xrightarrow{\partial_{-d+1}^F}  F_{-d} \to 0)$ be a minimal free resolution of $D$, so there is a quasi-isomorphism $F\to D$ and $F_i$ are finitely generated free $R$-modules for all $i\in\mathbb{Z}$. Set $A=(0\to F_0 \to \cdots \to F_{-d}\to 0)$ and $M=\cok(\partial_2^F)$. Since $H_i(F)=H_i(D)=0$ for $i>0$, by taking a hard truncation 
\[
  \xymatrix@C=20pt@R=15pt
  {
    A = ( 
    \cdots \ar[r]
    & 0 \ar[r] \ar[d]
    & 0 \ar[r] \ar[d]
    & F_0 \ar[r]^{\partial_0^F} \ar@{=}[d]
    & \cdots \ar[r]^{\partial_{-d+1}^F}
    & F_{-d} \ar[r] \ar@{=}[d] 
    & 0 )\\
    F= (\cdots \ar[r]
    & F_2 \ar[r]^{\partial_2^F} \ar@{=}[d]
    & F_1 \ar[r]^{\partial_1^F} \ar@{=}[d]
    & F_0 \ar[r]^{\partial_0^F} \ar[d]
    & \cdots \ar[r]^{\partial_{-d+1}^F}
    & F_{-d} \ar[r] \ar[d] 
    & 0 )\\
    M[1] \simeq (\cdots \ar[r]
    & F_2 \ar[r]^{\partial_2^F}
    & F_1 \ar[r]
    & 0 \ar[r]
    & \cdots \ar[r]
    & 0 \ar[r]
    & 0 ),\\
  }
\]
we obtain an exact triangle $M\to A\to D\to M[1]$. Since $A$ is a bounded complex of finitely generated free $R$-modules, $\Gid_R X<\infty$ gives $\Gid_R(A\otimes_R^{\L} X)<\infty$, and it follows from $\Gfd_R X<\infty$ and Foxby equivalence that $\Gid_R(D\otimes_R^{\L} X)<\infty$; see Remark \ref{Well-known fact}. Applying the two-out-of-three property to the exact triangle $M\otimes_R^{\L} X\to A\otimes_R^{\L} X\to D\otimes_R^{\L} X\to (M\otimes_R^{\L} X)[1]$, we obtain $\Gid_R(M\otimes_R^{\L} X)<\infty$ again by Remark \ref{Well-known fact}. In particular, $M\otimes_R^{\L} X$ is homologically bounded.

Let $I=(0\to I_s \xrightarrow{\partial_{s}^I} I_{s-1} \xrightarrow{\partial_{s-1}^I} \cdots)$ be an injective resolution of $M\otimes_R^{\L} X$, that is, there is a quasi-isomorphism $M\otimes_R P=M\otimes_R^{\L} X\to I$ and $I_i$ are injective for all $i\in\mathbb{Z}$. Since $P_i=0$ for $i<0$, one has $H_i(I)\cong H_i(M\otimes_R^{\L} X)=H_i(M\otimes_R P)=0$ for $i<0$. Putting $B=(0\to I_s \to \cdots \to I_1\to 0)$ and $L=\ker(\partial_0^I)$, we obtain an exact triangle $L\to M\otimes_R^{\L} X\to B\to L[1]$ by taking a hard truncation
\[
  \xymatrix@C=20pt@R=15pt
  {
    L \simeq (0 \ar[r]
    & 0 \ar[r] \ar[d]
    & \cdots \ar[r]
    & 0 \ar[r] \ar[d]
    & I_0 \ar[r]^{\partial_0^I} \ar@{=}[d]
    & I_{-1} \ar[r] \ar@{=}[d] 
    & \cdots )\\
    I= (0 \ar[r]
    & I_s \ar[r]^{\partial_s^I} \ar@{=}[d]
    & \cdots \ar[r]^{\partial_2^I} 
    & I_1 \ar[r]^{\partial_1^I} \ar@{=}[d]
    & I_0 \ar[r]^{\partial_0^I} \ar[d]
    & I_{-1} \ar[r] \ar[d] 
    & \cdots )\\
    B = (0 \ar[r]
    & I_s \ar[r]^{\partial_s^I}
    & \cdots \ar[r]^{\partial_2^I}
    & I_1 \ar[r]
    & 0 \ar[r]
    & 0 \ar[r]
    & \cdots ).\\
  }
\]
As $B$ is a bounded complex of injective $R$-modules, one has $\id_R B<\infty$. Combining this with $\Gid_R(M\otimes_R^{\L} X)<\infty$, we get $\Gid_R L<\infty$; see Remark \ref{Well-known fact}. By \cite[Theorem 4.4]{CFrH}, $L$ belongs to the Bass class for the dualizing complex $D$. Adopting the notation of \cite[Proposition 4.8(c)]{Chr}, we have ${\rm inf} L\ge 0$, ${\rm sup}D=0$, and $0\le {\rm inf} L-{\rm sup} D \le {\rm inf}(\RHom_R(D,L))$. In particular, the natural isomorphism $F^\ast \otimes_R L\cong\Hom_R(F,L)$ shows that $H_{-1}(F^\ast \otimes_R L)\cong H_{-1}(\Hom_R(F,L))=H_{-1}(\RHom_R(D,L))=0$ holds, which implies that
$$
(F_0)^\ast \otimes_R L \xrightarrow{(\partial_1^F)^\ast \otimes L} (F_1)^\ast \otimes_R L\xrightarrow{(\partial_2^F)^\ast \otimes L}  (F_2)^\ast \otimes_R L
$$
is exact. The exact sequence $ (F_1)^\ast\xrightarrow{(\partial_2^F)^\ast}  (F_2)^\ast\to \tr M\to 0$ induces a commutative diagram
\[
  \xymatrix@C=25pt@R=25pt
  {
    & (F_0)^\ast \otimes_R L  \ar[r] \ar[rd]^0 
    & (F_1)^\ast \otimes_R L \ar[r] \ar@{->>}[d]
    & (F_2)^\ast \otimes_R L \ar[r] \ar@{=}[d]
    & \tr M\otimes_R L \ar[r] \ar@{=}[d]
    & 0 \\
    0  \ar[r]
    & \Tor_1^R(\tr M, L)  \ar[r]
    & \image(\partial_2^F)^\ast\otimes_R L  \ar[r]^\theta
    & (F_2)^\ast \otimes_R L \ar[r]
    & \tr M\otimes_R L \ar[r] 
    & 0
  }
\]
with exact rows. A diagram chase shows that $\Tor_1^R(\tr M, L)\cong\ker\theta=0$.

We put $N=\cok(P_1\to P_0)$ and show that $R$-modules $L$, $M$, and $N$, together with the composition
$$
\pi: L\to H_0(I)\cong H_0(M\otimes_R P)=\cok(M\otimes_R P_1\to M\otimes_R P_0)\cong M\otimes_R \cok(P_1\to P_0)=M\otimes_R N
$$ 
of the natural epimorphism $L\to H_0(I)$ with the inverse of the isomorphism induced by the above quasi-isomorphism $M\otimes_R P=M\otimes_R^{\L} X \to I$ satisfy the hypotheses of Lemma \ref{apdtfreelem}. By the hypotheses of the theorem and the preceding argument, we have $N\otimes_R k\cong H_0(X)\otimes_R k\ne 0$ and $\Tor_1^R(\tr M, L)=0$. Let $\overline{p}\in H_0(P)$ be represented by an element $p\in P_0$. Then the natural homomorphism $M\to M\otimes_R P_0:m\mapsto m\otimes p$ induces a homomorphism $\alpha: M\to M\otimes_RP$ of complexes. The induced map on homology is the morphism $\phi:M\to M\otimes_R N$ appearing in Lemma \ref{apdtfreelem}. On the other hand, the composition $M\to I$ of the quasi-isomorphism $M\otimes_R P\to I$ with $\alpha$ is a morphism of complexes, and hence its image is contained in $L$. The induced map $\psi\colon M\to L$ satisfies $\phi=\pi \circ\psi$.

Lemma \ref{apdtfreelem} shows that $M$ is free. Since $A$ is a bounded complex of finitely generated free $R$-modules, the exact triangle $M\to A\to D\to M[1]$ yields $\pd_R D<\infty$, which means that $R$ is Gorenstein.
\end{proof}

As a consequence of Theorem \ref{GfdGid}, we obtain the following corollary. It gives an affirmative answer to Question \ref{quesmod}. Moreover, it provides a complex analogue of Proposition \ref{fdid}, thereby bringing this characterization of Gorenstein rings to its natural level of generality.

\begin{cor}\label{main cor}
Let $(R,\m,k)$ be a local ring. If there exists a complex $X$ of $R$-modules with $\Gfd_R X<\infty$,  $\Gid_R X<\infty$ and $\depth_R X<\infty$, then $R$ is Gorenstein.
\end{cor}

\begin{proof}
Let $\xx=x_1,\ldots,x_e$ be a system of generators of $\m$, and let $K=K(\xx)$ be a Koszul complex with respect to $\xx$. Put $Y=X\otimes_R^{\L} K$. As each $x_i$ acts null-homotopically on the Koszul complex $K$, $x_i$ annihilates $H_j(Y)$ for all $i$ and $j$. So all homologies of $Y$ are $k$-vector spaces. Since $K$ is a bounded complex of finitely generated free $R$-modules, tensoring with $K$ preserves the finiteness of Gorenstein flat and injective dimensions. Hence $\Gfd_R Y<\infty$ and $\Gid_R Y<\infty$ by assumption. It follows from \cite[Corollary 12.3.23]{CFoH} that there is a quasi-isomorphism 
$$\RHom_R(k,Y)=\RHom_R(k, X\otimes_R^{\L} K)\simeq \RHom_R(k,X)\otimes_R^{\L} K \simeq \RHom_R(k,X)\otimes_k^{\L} (k\otimes_R^{\L} K).$$
Since $\depth_R X<\infty$, $\RHom_R(k,X)$ is not acyclic. As $\RHom_R(k,X)\otimes_k^{\L} (k\otimes_R^{\L} K)$ is a derived tensor product of two non-acyclic complexes over the field, it is not acyclic; hence neither is $Y$. 

Let $\widehat{R}$ denote the $\m$-adic completion of $R$. We regard $Z=\mathbf{R}\Gamma_{\m}(Y)$ as a complex over $\widehat{R}$; this is justified by \cite[Corollary 13.4.17]{CFoH}. As all homologies of $Y$ are $k$-vector spaces, $Y$ is derived $\m$-torsion; hence by \cite[Theorem 13.4.9]{CFoH}, $Z$ is quasi-isomorphic to $Y$ as a complex over $R$. We see that $Z$ is not acyclic and that all homologies of $Z$ are $k$-vector spaces. Since $\operatorname{Gfd}_R Y<\infty$, \cite[Proposition 19.3.19]{CFoH} gives $ \operatorname{Gfd}_{\widehat{R}} Z<\infty$. 
Also, $Z\simeq \mathbf{L}\Lambda_{\m}(Z)$ as a complex over $\widehat{R}$; see \cite[Theorem 13.1.32]{CFoH}. It follows from \cite[Proposition 19.2.31]{CFoH} that $\operatorname{Gid}_{\widehat{R}} Z=\operatorname{Gid}_{\widehat{R}}\mathbf{L}\Lambda_{\m}(Z)\le \operatorname{Gid}_R Z=\operatorname{Gid}_R Y<\infty$.
Since all homologies of $Z$ are $k$-vector spaces, it follows that $H_n(Z)\otimes_{\widehat{R}} k=H_n(Z)\ne 0$ for $n={\rm inf}\{i \mid H_i(Z)\ne 0\}$. Theorem \ref{GfdGid} shows that $\widehat{R}$ is Gorenstein because it admits a dualizing complex, and hence so is $R$.
\end{proof}


\end{document}